\pdfoutput=1

\documentclass[12pt]{amsart}
\usepackage{color,amscd,epsfig,hhline,booktabs,amsmath,
latexsym,graphicx,mathtools,array,marginnote,url,enumitem,
mathrsfs,xfrac,bm,bbm,nccmath,multirow,tikz,tikz-cd,diagbox,nccmath}

\usepackage[T1]{fontenc}
\usepackage[utf8]{inputenc}
\usepackage{etex}
\usepackage{hyperref}
\usepackage[all]{xy}
\usepackage[protrusion=true,expansion=true]{microtype}
\usepackage[percent]{overpic}

\usepackage[margin=38mm]{geometry}
\usepackage{lipsum}

\tikzset{>=latex,auto}
\usetikzlibrary{positioning,shapes.geometric}

\newtheorem{theorem}{Theorem}[section]
\newtheorem{lemma}[theorem]{Lemma}
\newtheorem{corollary}[theorem]{Corollary}

\theoremstyle{definition}

\newtheorem{example}[theorem]{Example}

\theoremstyle{remark}
\newtheorem{remark}[theorem]{Remark}

\numberwithin{equation}{section}

\newcommand{\thmref}[1]{Theorem~\ref{#1}}

\newcommand{\remref}[1]{Remark~\ref{#1}}

\newcommand{\lemref}[1]{Lemma~\ref{#1}}

\newcommand{\corref}[1]{Corollary~\ref{#1}}
\newcommand{\figref}[1]{Fig.~\ref{#1}}
\newcommand{\exref}[1]{Example~\ref{#1}}

\newcommand{\vs}{\vspace{6pt}}

\newcommand{\RR}{{\mathbb R}}

\newcommand{\ZZ}{{\mathbb Z}}

\newcommand{\OO}{{\mathscr O}}
\newcommand{\sS}{{\mathscr S}}
\newcommand{\sR}{{\mathscr R}}
\newcommand{\sC}{{\mathscr C}}

\newcommand{\modd}{\ \ (\operatorname{mod} \ \ZZ)}
\newcommand{\con}{\operatorname{const.}}

\newcommand{\mb}{\mathbf{m}_3}
\newcommand{\mc}{\mathbf{m}_4}
\newcommand{\mk}{\mathbf{m}_k}
\newcommand{\des}{\operatorname{des}}
\newcommand{\sig}{\operatorname{sig}}
\newcommand{\dep}{\operatorname{dep}}

\newcommand{\sym}{\operatorname{sym}}

\newcommand{\ve}{\varepsilon}

\newcommand{\knuth}{\genfrac{\langle }{\rangle}{0pt}{}}
\newcommand{\mo}{\operatorname{mod}}

\newcommand{\bit}{\it \bfseries}

\begin{document}
\title[Cyclic permutations]{Cyclic Permutations: Degrees and Combinatorial Types}

\author[S. Zakeri]{Saeed Zakeri}

\address{Department of Mathematics, Queens College of CUNY, 65-30 Kissena Blvd., Queens, New York 11367, USA} 
\address{The Graduate Center of CUNY, 365 Fifth Ave., New York, NY 10016, USA}

\email{saeed.zakeri@qc.cuny.edu}

\date{May 31, 2021}

\begin{abstract}
This note will give an enumeration of \mbox{$n$-cycles} in the symmetric group $\sS_n$ by their degree (also known as their cyclic descent number) and studies similar counting problems for the conjugacy classes of $n$-cycles under the action of the rotation subgroup of $\sS_n$. This is achieved by relating such cycles to periodic orbits of an associated dynamical system acting on the circle. We also compute the mean and variance of the degree of a random $n$-cycle and show that its distribution is asymptotically normal as $n \to \infty$.   
\end{abstract}

\maketitle

\section{Introduction}

The classical Eulerian numbers describe the distribution of descent number in the full symmetric group $\sS_n$ and have been studied extensively for more than a century (see for example \cite{P2} and \cite{St}). Understanding the distribution of descent number in a given conjugacy class of $\sS_n$ is a more subtle problem that was first tackled in the 1990's by Gessel and Reutenauer \cite{GR}, by Diaconis, McGrath, and Pitman \cite{DMP}, and by Fulman \cite{F1}.  \vs

This note will consider a variant of the descent number of a permutation $\nu \in \sS_n$, which we call its {\it degree}, defined by 
$$
\deg(\nu) = \# \big\{ i : \nu(i) > \nu(i+1) \big\}, 
$$
where the integer $i$ is taken modulo $n$. Our terminology is justified by a simple topological interpretation of this quantity, but it turns out that what we call degree in this paper has already been studied in the combinatorics literature under the name {\it cyclic descent number} (see for example \cite{C}, \cite{F2}, \cite{P1}, and \cite{DPS}). The degree has the advantage of being invariant under the left and right actions of the rotation subgroup of $\sS_n$ generated by the cycle $(1 \, 2 \, \cdots \, n)$, and naturally occurs in the study of the combinatorial patterns of periodic orbits of covering maps of the circle (see \cite{M} and \cite{PZ}). Motivated by this connection, we will investigate the distribution of degree in the conjugacy class $\sC_n$ of all $n$-cycles in $\sS_n$. Let $N_{n,d}$ denote the number of $\nu \in \sC_n$ with $\deg(\nu)=d$. In \S \ref{sec:des} we prove 

\begin{theorem}\label{YEK}
For every $d \geq 1$,
$$
N_{n,d} = \sum_{k=1}^d (-1)^{d-k} \binom{n}{d-k} \Delta_n(k),
$$
where 
$$
\Delta_n(k) = \sum_{r|n} \mu \Big( \frac{n}{r} \Big) \Big(\sum_{j=0}^{r-1} k^j \Big)
$$
and $\mu$ is the M\"{o}bius function. 
\end{theorem}

This is the analog of the alternating sum formula for Eulerian numbers, and the similar formulas in \cite{DMP} and \cite{F1} for permutations with a given cycle structure and descent number. Our proof makes essential use of a count for the number of period $n$ orbits of the linear endomorphism $\mk(x)=kx \ (\mo \ZZ)$ of the circle $\RR/\ZZ$ that realize the combinatorics of degree $d$ elements in $\sC_n$, developed by C.~L.~Petersen and the author \cite{PZ}. The $\Delta_n(k)$ for $k \geq 2$ can be interpreted as the number of period $n$ points of $\mk$ up to rotation by a \mbox{$(k-1)$-st} root of unity (see \S \ref{keek}, \S \ref{CNQD} and \S \ref{compp}). \vs 

It follows immediately from \thmref{YEK} that the generating function $G_n(x)= \sum_{d=1}^{n-2} N_{n,d} \, x^d$ has the power series expansion 
$$
G_n(x)=(1-x)^n \sum_{i \geq 1} \Delta_n(i) \, x^i.
$$
This result should be viewed as the analog of Carlitz's identity for Eulerian numbers (compare \S \ref{CNQD} and \S \ref{compp}). \vs

It is well known that the descent number of a randomly chosen permutation in $\sS_n$ has mean $(n-1)/2$ and variance $(n+1)/12$ (see for example \cite{Pi}). In \S \ref{sec:stat} we prove  

\begin{theorem}\label{DO}
The degree of a randomly chosen cycle in $\sC_n$ (with respect to the uniform measure) has mean  
$$ 
\frac{n}{2}-\frac{1}{n-1} \qquad \text{if} \ n \geq 3,
$$
and variance
$$
\frac{n}{12}+\frac{n}{(n-1)^2(n-2)} \qquad \text{if} \ n \geq 5.
$$
\end{theorem}

The idea of the proof, inspired by the method of Fulman in \cite{F1}, is to express the generating functions $G_n$ in terms of the generating functions of the Eulerian numbers (the so-called Eulerian polynomials) for which the mean and variance are already known (see \S \ref{asnorm}). \vs 

The following central limit theorem for the degree is also proved in \S \ref{sec:stat}: 

\begin{theorem}\label{SE}
When normalized by its mean and variance, the distribution of $\deg(\nu)$ for $\nu \in \sC_n$ converges to the standard normal distribution as $n \to \infty$.   
\end{theorem}

Compare \figref{normal}. \vs

\begin{figure}[t]
\centering
\begin{overpic}[width=\textwidth]{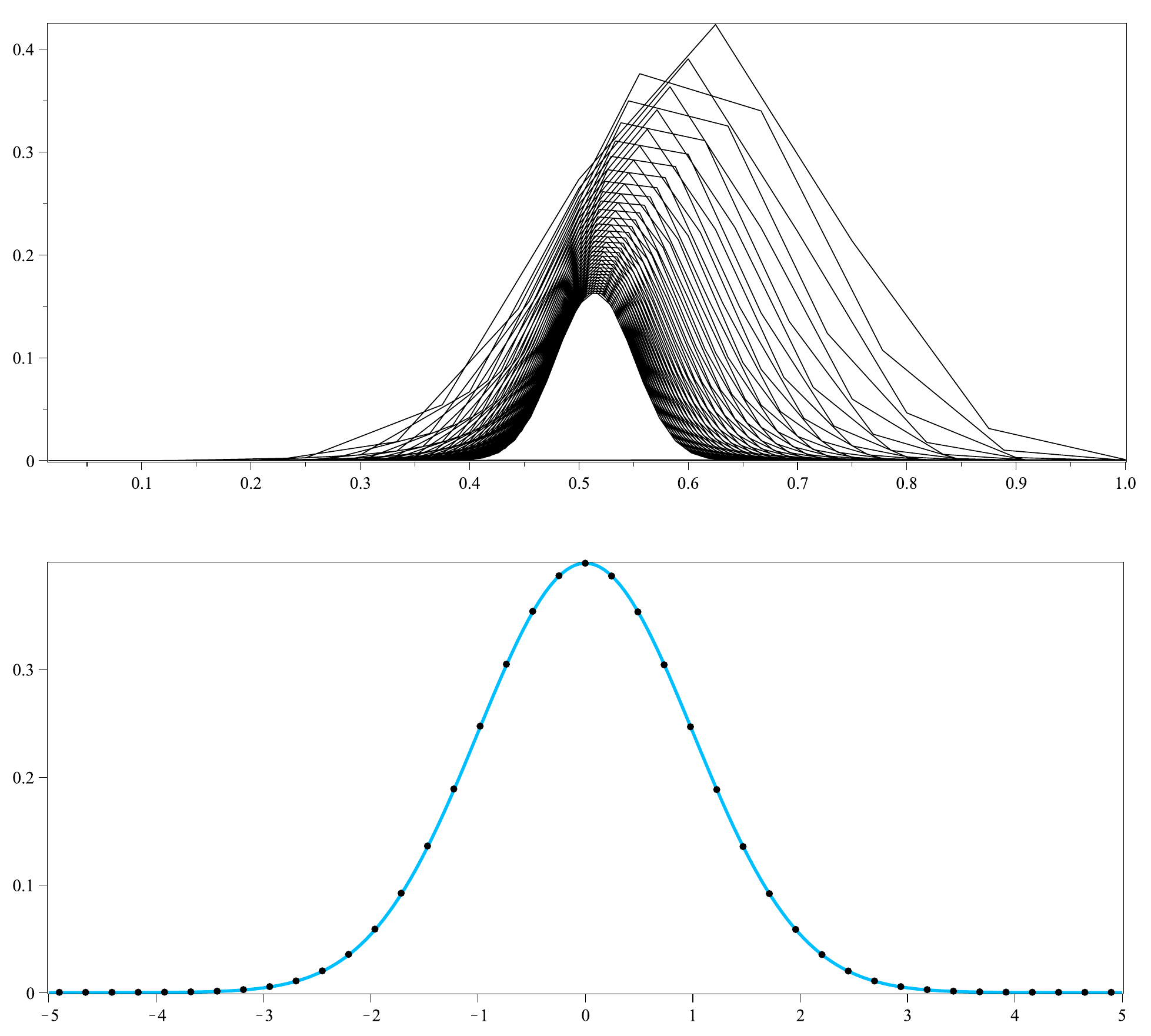}
\end{overpic}
\caption{\sl Illustration of \thmref{SE}. Top: The histogram of the distribution of degree among $n$-cycles, scaled horizontally to $n-2$, for $10 \leq n \leq 70$. The rapid convergence to a bell-shaped density is evident. Bottom: The distribution of degree among $200$-cycles, normalized by its mean and variance (only 41 out of 198 points are visible in the given window). The curve in blue is the standard normal distribution.}  
\label{normal}
\end{figure}

The central limit theorem for the distribution of descent number over $\sS_n$ or a given conjugacy class has been known (see for example the papers of Bender \cite{B} or Harper \cite{H} for the former, and the recent work of Kim and Lee \cite{KL1} for the latter). Our proof follows the strategy of \cite{FKL} and is based on the idea of reducing convergence in distribution to pointwise convergence of the moment generating functions over some non-empty open interval, as utilized in \cite{KL2} (see \S \ref{asnorm}). \vs

Motivated by applications in dynamics, we also study the conjugacy classes of $\sC_n$ under the action of the rotation subgroup of $\sS_n$. Each such class is called a {\it combinatorial type} in $\sC_n$. In \S \ref{sec:so} we count the number of $n$-cycles of a given symmetry order and use it to derive a (known) formula for the number of distinct combinatorial types in $\sC_n$ (compare Theorems \ref{NQR} and \ref{NQ}). This section is elementary and rather independent of the rest of the paper, except for \S \ref{ttnn} where we discuss the problem of counting the number of distinct combinatorial types of a given degree. \vs

A partial motivation for writing the present note was the belief that using dynamical systems could give a different and interesting perspective on the problem of enumeration of $n$-cycles by their degree. After a preliminary version of this paper appeared on arXiv.org, J. Fulman informed me that an alternative proof of \thmref{YEK} can be given using results in his paper [F2]. More generally, he is able to derive similar enumeration results in arbitrary conjugacy classes in $\sS_n$. However, his methods are not elementary or combinatorial, as they rely on representation theory of the Whitehouse module. An interesting open problem is to develop a purely combinatorial approach to our degree counts for $n$-cycles or other conjugacy classes in $\sS_n$. \vs

\noindent
{\it Acknowledgments.} I'm grateful to J. Fulman for his support and interest in this work, and for sharing his insights on the problems discussed here. I also thank anonymous referees for several helpful comments, which in particular brought my attention to some of the previous work on cyclic descents.     
  
\section{Preliminaries} \label{sec:pre}

Fix an integer $n \geq 2$. We denote by $\sS_n$ the group of all permutations of $\{ 1,\ldots, n \}$ and by $\sC_n$ the collection of all $n$-cycles in $\sS_n$. Following the tradition of group theory, we represent $\nu \in \sC_n$ by the symbol
$$
( 1 \ \nu(1) \ \nu^2(1) \ \cdots \ \nu^{n-1}(1) ). 
$$
The {\bit rotation group} $\sR_n$ is the cyclic subgroup of $\sS_n$ generated by the $n$-cycle 
$$
\rho := (1 \, 2 \, \cdots \, n). 
$$
Elements of $\sR_n \cap \sC_n$ are called {\bit rotation cycles}. Thus, $\nu \in \sC_n$ is a rotation cycle if and only if $\nu = \rho^m$ for some integer $1 \leq m < n$ with $\gcd(m,n)=1$. The reduced fraction $m/n$ is called the {\bit rotation number} of $\rho^m$. \vs

The rotation group $\sR_n$ acts on $\sC_n$ by conjugation. We refer to each orbit of this action as a {\bit combinatorial type} in $\sC_n$. The combinatorial type of an $n$-cycle $\nu$ is denoted by $[\nu]$. It is easy to see that $\nu$ is a rotation cycle if and only if $[\nu]$ consists of $\nu$ only. In fact, if $\rho \nu \rho^{-1} = \nu$, then $\nu=\rho^m$ where $m=\nu(n)$. 

\subsection{The symmetry order}

The combinatorial type of $\nu \in \sC_n$ can be described explicitly as follows. Let 
$$
G_\nu := \big\{ \rho^j : \rho^j \nu \rho^{-j}=\nu \big\}
$$ 
be the stabilizer group of $\nu$ under the action of $\sR_n$. We call the order of $G_\nu$ the {\bit symmetry order} of $\nu$ and denote it by $\sym(\nu)$. If $r:=n/\sym(\nu)$, it follows that $G_\nu$ is generated by the power $\rho^r$ and the combinatorial type of $\nu$ is the $r$-element set 
$$
[\nu]= \big\{ \nu, \rho \nu \rho^{-1}, \ldots, \rho^{r-1} \nu \rho^{-(r-1)} \big\}.
$$
Since $\sym(\rho \nu \rho^{-1})=\sym(\nu)$, we can define the symmetry order of a combinatorial type unambiguously as that of any cycle representing it:
$$
\sym([\nu]):=\sym(\nu). 
$$
Evidently there are no $2$- or $3$-cycles of symmetry order $1$, and there is no $4$-cycle of symmetry order $2$. By contrast, it is not hard to see that for every $n \geq 5$ and every divisor $s$ of $n$ there is a $\nu \in \sC_n$ with $\sym(\nu)=s$. \vs

Of the $(n-1)!$ elements of $\sC_n$, precisely $\varphi(n)$ are rotation cycles. Here $\varphi$ is Euler's totient function defined by 
$$
\varphi(n) := \# \, \{ m \in \ZZ: 1 \leq m \leq n \ \text{and} \ \gcd(m,n)=1 \}.  
$$
If $\nu_1, \ldots, \nu_T$ are representatives of the distinct combinatorial types in $\sC_n$, then
$$
(n-1)! = \sum_{\nu_i \in \sR_n } \# [ \nu_i] + \sum_{\nu_i \notin \sR_n } \# [ \nu_i] = \varphi(n)+ \sum_{\nu_i \notin \sR_n } \# [ \nu_i].
$$
When $n$ is a prime number, we have $\varphi(n)=n-1$ and each $\# [\nu_i]$ in the far right sum is $n$. In this case the number of distinct combinatorial types in $\sC_n$ is given by 
\begin{equation}\label{hayoola}
T = (n-1) + \frac{(n-1)!-(n-1)}{n} = \frac{(n-1)!+(n-1)^2}{n}.
\end{equation}
Observe that $T$ being an integer gives a simple proof of Wilson's theorem according to which $(n-1)! = -1 \ (\operatorname{mod} n)$ whenever $n$ is prime.
  
\begin{figure}[t]
\centering
\begin{overpic}[width=\textwidth]{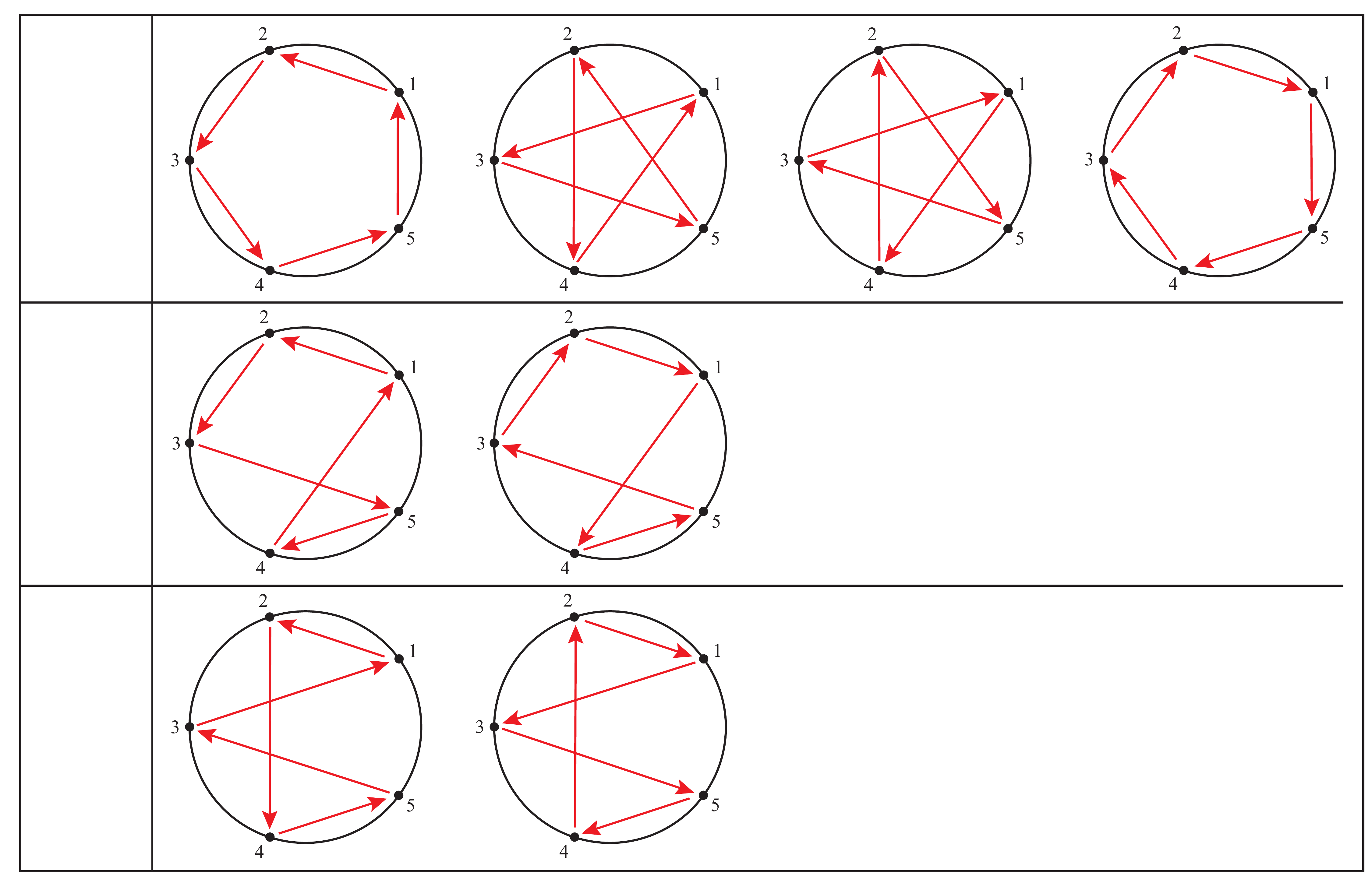}
\put (56,31) {$+$ four rotated copies of each}
\put (56,10) {$+$ four rotated copies of each}
\put (4,10) {$\sC_{5,3}^1$}
\put (4,31) {$\sC_{5,2}^1$}
\put (4,52) {$\sC_{5,1}^5$}
\put (21.5,52) {\footnotesize{$\rho$}}
\put (43.5,52) {\footnotesize{$\rho^2$}}
\put (65.5,52) {\footnotesize{$\rho^3$}}
\put (88,52) {\footnotesize{$\rho^4$}}
\put (21,33) {\footnotesize{$\pi$}}
\put (42,33) {\footnotesize{$\pi^{-1}$}}
\put (24,11) {\footnotesize{$\nu$}}
\put (46,11) {\footnotesize{$\nu^{-1}$}}
\end{overpic}
\caption{\sl The decomposition of $\sC_5$ into subsets $\sC_{5,d}^s$ of cycles with degree $d$ and symmetry order $s$, where the only admissible pairs are $(d,s)=(1,5), (2,1), (3,1)$. See Examples \ref{cyc5} and \ref{ajab}.}  
\label{c5}
\end{figure}

\begin{example}\label{cyc5}
The $4!=24$ cycles in $\sC_5$ fall into $(4!+4^2)/5=8$ distinct combinatorial types. The $4$ rotation cycles 
\begin{align*}
\rho & = (1 \, 2 \, 3 \, 4 \, 5) & \rho^2 & = (1 \, 3 \, 5 \, 2 \, 4) \\
\rho^3 & = (1 \, 4 \, 2 \, 5 \, 3) & \rho^4 & = (1 \, 5 \, 4 \, 3 \, 2) 
\end{align*}
(of rotation numbers $1/5, 2/5, 3/5, 4/5$) form $4$ distinct combinatorial types. The remaining $20$ cycles have symmetry order $1$, so they fall into $4$ combinatorial types each containing $5$ elements. These types are represented by
\begin{align*}
\pi & = (1 \, 2 \, 3 \, 5 \, 4) & \pi^{-1} & =(1 \, 4 \, 5 \, 3 \, 2) \\ 
\nu & =(1 \, 2 \, 4 \, 5 \, 3) &  \nu^{-1} & =(1 \, 3 \, 5 \, 4 \, 2). 
\end{align*}
Compare \figref{c5}.  
\end{example}

\subsection{Descent number vs. degree}\label{DNUM}

A permutation $\nu \in \sS_n$ has a {\bit descent} at $i \in \{ 1, \ldots, n-1 \}$ if $\nu(i)>\nu(i+1)$. The total number of such $i$ is called the {\bit descent number} of $\nu$:
$$
\des(\nu) := \# \big\{ 1 \leq i \leq n-1  : \nu(i) > \nu(i+1) \big\} 
$$
Note that $0 \leq \des(\nu) \leq n-1$. The descent number is a basic tool in enumerative combinatorics (see for example \cite{St}). \vs

In this paper we will be working with a rotationally invariant version of the descent number which we call {\bit degree}.\footnote{As noted in the introduction, our degree is synonymous to what combinatorists have called ``cyclic descent number'' in recent years. Somewhat unfortunately, this same invariant is referred to as ``descent number'' in \cite{PZ}.} It simply amounts to counting $i=n$ as a descent if $\nu(n)>\nu(1)$:
$$
\deg(\nu) := \begin{cases} \des(\nu) & \quad \text{if} \ \nu(n)<\nu(1) \\ \des(\nu)+1 & \quad \text{if} \ \nu(n)>\nu(1). \end{cases} 
$$
The terminology comes from the following topological characterization (see \cite{M} and \cite{PZ}): Take any set $\{ x_1, \ldots, x_n \}$ of distinct points on the circle in positive cyclic order. Then $\deg(\nu)$ is the minimum degree of a continuous covering map $f: \RR/\ZZ \to \RR/\ZZ$ which acts on this set as the permutation $\nu$ in the sense that $f(x_i)=x_{\nu(i)}$ for all $i$. A simple model which realizes this minimum degree is the covering map that sends each counter-clockwise arc $[x_i,x_{i+1}]$ affinely onto $[x_{\nu(i)},x_{\nu(i+1)}]$. \vs 

\begin{example}\label{ajab}
The cycle $\nu=(1 \, 2 \, 4 \, 5 \, 3) \in \sC_5$ has descents at $i=2$, $i=4$ and $i=5$, so $\deg(\nu)=3$. The eight representative cycles in $\sC_5$ described in \exref{cyc5} have the following degrees:
\begin{align*}
& \deg(\rho)=\deg(\rho^2)=\deg(\rho^3)=\deg(\rho^4)=1, \\
& \deg(\pi)=\deg(\pi^{-1})=2, \\
& \deg(\nu)=\deg(\nu^{-1})=3.
\end{align*}
The four cycles of degree $1$ have symmetry order $5$, and the remaining cycles have symmetry order $1$. Compare \figref{c5}.
\end{example} 

The following statement summarizes the basic properties of degree for cycles: 

\begin{theorem}\label{degp}
Let $\nu \in \sC_n$ with $\sym(\nu)=s$ and $\deg(\nu)=d$. \vs
\begin{enumerate}
\item[(i)]
$1 \leq d \leq n-2$ if $n \geq 3$. \vs
\item[(ii)]
$d=1 \Longleftrightarrow s=n \Longleftrightarrow \nu \in \sR_n$. \vs
\item[(iii)]
$s | \gcd(n,d-1)$. \vs
\item[(iv)]
$\deg(\rho \nu) = \deg(\nu \rho) = \deg(\rho \nu \rho^{-1})=d$.  
\end{enumerate}
\end{theorem}

Apart from (iii) the other claims are rather straightforward; see \cite[Lemma 2.5, Theorem 2.8, and Theorem 6.4]{PZ}. I'm informed by a referee that parts (ii) and (iv) also appear in \cite[Lemma 1]{K}. \vs 

Observe that by (iv), the degree of a combinatorial type is well-defined:
$$
\deg([\nu]):=\deg(\nu).
$$

\begin{example}\label{hasht}
Take $\nu \in \sC_8$ with $\sym(\nu)=s$ and $\deg(\nu)=d$. By the above theorem $1 \leq d \leq 6$, with $d=1$ if and only if $s=8$. Moreover, $s$ is a common divisor of $8$ and $d-1$. It follows that the only admissible pairs $(d,s)$ are $(1,8), (2,1), (3,1), (3,2), (4,1), (5,1), (5,2), (5,4), (6,1)$.   
\end{example} 

\subsection{Decompositions of $\sC_n$}

Fix $n \geq 3$ and consider the following cross sections of $\sC_n$ by the symmetry order and degree:
\begin{align*}
\sC_n^s & := \{ \nu \in \sC_n : \sym(\nu)=s \} \\
\sC_{n,d} & := \{ \nu \in \sC_n : \deg(\nu)=d \} \\
\sC_{n,d}^s & := \sC_n^s \cap \sC_{n,d}.
\end{align*}
Observe that in our notation the symmetry order always appears as a superscript and the degree as a subscript after $n$. By \thmref{degp}, 
$$
\sC_n^n = \sC_{n,1} = \sC_{n,1}^n = \sC_n \cap \sR_n
$$
and we have the decompositions  
\begin{align*}
\sC_n & = \bigcup_{s|n} \sC_n^s = \bigcup_{d=1}^{n-2} \sC_{n,d} & & \\
\sC_n^s & = \bigcup_{j=1}^{\lfloor (n-3)/s \rfloor} \sC_{n,js+1}^s & & \text{if} \ s|n, \ s<n \\
\sC_{n,d} & = \bigcup_{s|\gcd(n,d-1)} \sC_{n,d}^s & & \text{if} \ 2 \leq d \leq n-2.
\end{align*}
Hence the cardinalities
\begin{align*}
N_n^s & := \# \, \sC_n^s \\
N_{n,d} & := \# \, \sC_{n,d} \\
N_{n,d}^s & := \# \, \sC_{n,d}^s
\end{align*}
satisfy the following relations: 
\begin{align}
N_n^n & = N_{n,1} = N_{n,1}^n = \varphi(n) & & \notag\\ 
(n-1)! & = \sum_{s|n} N_n^s = \sum_{d=1}^{n-2} N_{n,d} \notag \\
N_n^s & = \sum_{j=1}^{\lfloor (n-3)/s \rfloor} N_{n,js+1}^s & & \text{if} \ s|n, \ s<n \label{foo} \\
N_{n,d} & = \sum_{s|\gcd(n,d-1)} N_{n,d}^s & & \text{if} \ 2 \leq d \leq n-2. \label{loo}  
\end{align}
Let us also consider the counts for the corresponding combinatorial types
\begin{align*}
T_n & := \# \, \{ [\nu]: \nu \in \sC_n \} \\
T_n^s & := \# \, \{ [\nu]: \nu \in \sC_n^s \} \\
T_{n,d} & := \# \, \{ [\nu]: \nu \in \sC_{n,d} \} \\
T_{n,d}^s & := \# \, \{ [\nu]: \nu \in \sC_{n,d}^s \}. 
\end{align*}
Evidently
$$
T_{n,d}^s = \frac{s}{n} \ N_{n,d}^s \qquad \text{and} \qquad T_n^s = \frac{s}{n} \ N_n^s
$$
and we have the following relations: 
\begin{align}
T_n^n & = T_{n,1} =T_{n,1}^n = \varphi(n) \notag\\ 
T_n & = \frac{1}{n} \sum_{s|n} s N_n^s \label{soo} \\
T_{n,d} & = \frac{1}{n} \sum_{s|\gcd(n,d-1)} s N_{n,d}^s \qquad \text{if} \ 2 \leq d \leq n-2.\notag
\end{align}
Of course knowing the joint distribution $N_{n,d}^s$ would allow us to count all the $N$'s and $T$'s. However, finding an closed formula for $N_{n,d}^s$ seems to be difficult (a sample computation can be found in \S \ref{ttnn}). In \S \ref{subsec:sym} we derive a formula for $N_n^s$ by a direct count which in turn leads to a formula for $T_n$ (see Theorems \ref{NQR} and \ref{NQ}). In \S \ref{CNQD} we find a formula for $N_{n,d}$ indirectly by relating cycles in $\sC_{n,d}$ to periodic orbits of an associated dynamical system acting on the circle (see \thmref{jeeg}). 

\section{The distribution of symmetry order} \label{sec:so}

\subsection{The numbers $N_n^s$}\label{subsec:sym} 

We begin with the simplest of our counting problems, that is, finding a formula for $N_n^s$. We will make use of the M\"{o}bius inversion formula
\begin{equation}\label{MIF}
g(m) = \sum_{k|m} f(k) \quad \Longleftrightarrow \quad f(m)= \sum_{k|m} \mu(k) \, g\Big( \frac{m}{k} \Big)
\end{equation}
on a pair of arithmetical functions $f,g$. Here $\mu$ is the M\"{o}bius function uniquely determined by the conditions $\mu(1):=1$ and $\sum_{k|m} \mu(k)=0$ for $m>1$. Applying \eqref{MIF} to the relation
$$
m = \sum_{k|m} \varphi(k)
$$ 
gives the classical identity
\begin{equation}\label{muu}
\varphi(m) = \sum_{k|m} \frac{m}{k} \, \mu(k) = \sum_{k|m} k \mu\Big( \frac{m}{k} \Big).
\end{equation}

\begin{theorem}\label{NQR}
For every $n \geq 2$ and every divisor $s$ of $n$,
\begin{equation}\label{ninu}
N_n^s = \frac{1}{n} \sum_{j|\frac{n}{s}} \mu (j) \, \varphi(sj) \, (sj)^{\frac{n}{sj}} \Big( \frac{n}{sj} \Big)!
\end{equation}
\end{theorem}

When $s=n$ the formula reduces to $N_n^n= (1/n) \mu(1) \varphi(n) n = \varphi(n)$ which agrees with our earlier count. 

\begin{proof}
Set $r:=n/s$. We have $\rho^r \nu \rho^{-r}=\nu$ if and only if $\sym(\nu)$ is a multiple of $s$ if and only if $\nu \in \sC_n^{n/j}$ for some $j|r$. Denoting $\nu$ by $(\nu_1 \ \nu_2 \ \cdots \ \nu_n)$, this condition can be written as 
$$
(\rho^r(\nu_1) \ \rho^r(\nu_2) \ \cdots \ \rho^r(\nu_n)) = (\nu_1 \ \nu_2 \ \cdots \ \nu_n), 
$$
which holds if and only if there is an integer $m$ such that 
\begin{equation}\label{yek}
\rho^r(\nu_i) = \nu_{\rho^m(i)} \qquad \text{for all} \ i.
\end{equation}
The rotations $\rho^r: i \mapsto i+r$ and $\rho^m: i \mapsto i+m \  (\mo n)$ have orders $n/\gcd(r,n)=n/r$ and $n/\gcd(m,n)$ respectively. By \eqref{yek}, these orders are equal, hence 
$$
r=\gcd(m,n).
$$
Setting $t:=m/r$ gives $\gcd(t,s)=1$, so there are at most $\varphi(s)$ possibilities for $t$ and therefore for $m$. The action of the rotation $\rho^m$ partitions $\ZZ/ n\ZZ$ into $r$ disjoint orbits each consisting of $s$ elements and these $r$ orbits are represented by $1, \ldots, r$. In fact, if 
$$
i+ \ell m = i'+ \ell' m \ (\mo n) \quad \text{for some} \ 1 \leq i,i' \leq r \ \text{and} \ 1 \leq \ell, \ell' \leq s,
$$  
then $i-i' = m(\ell'-\ell) \ (\mo n)$ so $i = i' \ (\mo r)$ which gives $i=i'$. Moreover, $\ell m =  \ell' m \ (\mo n)$ so $\ell t = \ell' t \ (\mo s)$. Since $\gcd(t,s)=1$, this implies $\ell = \ell' \ (\mo s)$ which shows $\ell=\ell'$. \vs

Now \eqref{yek} shows that for each of the $\varphi(s)$ choices of $m$, the cycle $\nu$ is completely determined by the integers $\nu_1, \ldots, \nu_r$, and different choices of $m$ lead to different cycles. We may always assume $\nu_1=1$. This leaves $n-s$ choices for $\nu_2$ (corresponding to the elements of $\{ 1, \ldots, n \}$ that are not in the orbit of $\nu_1=1$ under $\rho^m$), $n-2s$ choices for $\nu_3$, $\ldots$ and $n-(r-1)s=s$ choices for $\nu_r$. Thus, the total number of choices for $\nu$ is
$$
\varphi(s) (n-s)(n-2s) \cdots s = \varphi(s)\, s^{r-1} \, (r-1)! = \frac{1}{n} \varphi(s) \, s^r \, r!
$$  
This proves 
$$
\sum_{j|r} N_n^{n/j} = \frac{1}{n} \varphi \Big( \frac{n}{r} \Big) \Big( \frac{n}{r} \Big)^r r!
$$      
An application of the M\"{o}bius inversion formula \eqref{MIF} then gives
\begin{align*}
N_n^s = N_n^{n/r} & = \frac{1}{n} \sum_{j|r} \mu (j) \, \varphi\Big( \frac{nj}{r} \Big) \Big( \frac{nj}{r} \Big)^{\frac{r}{j}} \Big( \frac{r}{j} \Big)! \\
& =\frac{1}{n} \sum_{j|\frac{n}{s}} \mu (j) \, \varphi(sj) \, (sj)^{\frac{n}{sj}} \Big( \frac{n}{sj} \Big)! \qedhere
\end{align*}
\end{proof}

Table \ref{tab1} shows the values of $N_n^s$ for $2 \leq n \leq 15$. Notice that $N_2^1=N_3^1=N_4^2=0$ but all other values are positive. Moreover, as $n$ gets larger the distribution $N_n^s$ appears to be overwhelmingly concentrated at $s=1$. In fact, an elementary exercise gives the asymptotic estimate $N_n^1 \sim (n-1)!$ as $n \to \infty$ (compare \thmref{asymp} below for a similar analysis). This justifies the intuition that the chance of a randomly chosen $n$-cycle having any non-trivial rotational symmetry tends to zero as $n \to \infty$. 

{\tiny
\begin{table}
\centering
\begin{tabular}{r|crrrrrrrrrrrrrrr}
\toprule
\diagbox{$n$}{$s$} & & $1$ & $2$ & $3$ & $4$ & $5$ & $6$ & $7$ & $8$ & $9$ & $10$ & $11$ & $12$ & $13$ & $14$ & $15$ \\ 
\midrule
$2$ & & $0$ & $1$ & & & & & & & & & & & & & \\  
$3$ & & $0$ & $-$ & $2$ & & & & & & & & & & & & \\  
$4$ & & $4$ & $0$ & $-$ & $2$ & & & & & & & & & & & \\  
$5$ & & $20$ & $-$ & $-$ & $-$ & $4$ & & & & & & & & & & \\  
$6$ & & $108$ & $6$ & $4$ & $-$ & $-$ & $2$ & & & & & & & & & \\  
$7$ & & $714$ & $-$ & $-$ & $-$ & $-$ & $-$ & $6$ & & & & & & & & \\  
$8$ & & $4992$ & $40$ & $-$ & $4$ & $-$ & $-$ & $-$ & $4$ & & & & & & &\\  
$9$ & & $40284$ & $-$ & $30$ & $-$ & $-$ & $-$ & $-$ & $-$ & $6$ & & & & & &\\  
$10$ & & $362480$ & $380$ & $-$ & $-$ & $16$ & $-$ & $-$ & $-$ & $-$ & $4$ & & & & & \\ 
$11$ & & $3628790$ & $-$ & $-$ & $-$ & $-$ & $-$ & $-$ & $-$ & $-$ & $-$ & $10$ & & & & \\ 
$12$ & & $39912648$ & $3768$ & $312$ & $60$ & $-$ & $8$ & $-$ & $-$ & $-$ & $-$ &  $-$ & $4$ & & & \\ 
$13$ & & $479001588$ & $-$ & $-$ & $-$ & $-$ & $-$ & $-$ & $-$ & $-$ & $-$ & $-$ & $-$ & $12$ & & \\ 
$14$ & & $6226974684$ & $46074$ & $-$ & $-$ & $-$ & $-$ & $36$ & $-$ & $-$ & $-$ & $-$ & $-$ & $-$ & $6$ & \\ 
$15$ & & $87178287120$ & $-$ & $3880$ & $-$ & $192$ & $-$ & $-$ & $-$ & $-$ & $-$ & $-$ & $-$ & $-$ & $-$ & $8$ \\
\bottomrule 
\end{tabular}
\vspace*{2mm}
\caption{\sl The distributions $N_n^s$ for $2 \leq n \leq 15$.}
\label{tab1}
\end{table}
}

\subsection{The numbers $T_n$} 

The count \eqref{ninu} led us to the following formula for the number of distinct combinatorial types of $n$-cycles. It turns out that this formula is not new: It appears in the {\it On-line Encyclopedia of Integer Sequences} as the number of $2$-colored patterns of an $n \times n$ chessboard \cite{Sl}.  
 
\begin{theorem}\label{NQ}
For every $n \geq 2$, 
\begin{equation}\label{tq}
T_n = \frac{1}{n^2} \sum_{j|n} (\varphi(j))^2 \, j^{\frac{n}{j}} \Big( \frac{n}{j} \Big)! 
\end{equation}
\end{theorem}

Observe that for prime $n$ the formula reduces to  
$$
T_n = \frac{1}{n^2} \big( (\varphi(1))^2 \, n! + (\varphi(n))^2 \, n \big) = \frac{1}{n} ( (n-1)! + (n-1)^2 )
$$
which agrees with our derivation in \eqref{hayoola}. Table 
\ref{teeq} shows the values of $T_n$ for $2 \leq n \leq 15$.

\begin{proof}
By \eqref{soo} and \eqref{ninu}, 
$$
T_n = \frac{1}{n} \sum_{s|n} s N_n^s = \frac{1}{n^2} \sum_{s|n} \sum_{j|\frac{n}{s}} s \mu(j) \, \varphi(sj) \, (sj)^{\frac{n}{sj}} \Big( \frac{n}{sj} \Big)!
$$
The sum interchange formula
$$
\sum_{s|n} \sum_{j|\frac{n}{s}} f(j,s) = \sum_{j|n} \sum_{s|j} f\Big( \frac{j}{s},s \Big) 
$$
then gives 
\begin{align*}
T_n & = \frac{1}{n^2} \sum_{j|n} \sum_{s|j} s \mu\Big( \frac{j}{s} \Big) \varphi(j) \, j^{\frac{n}{j}} \Big( \frac{n}{j} \Big)!  \\
& = \frac{1}{n^2} \sum_{j|n} \Big( \sum_{s|j} s \mu\Big( \frac{j}{s} \Big) \Big) \, \varphi(j) \, j^{\frac{n}{j}} \Big( \frac{n}{j} \Big)! \\
& = \frac{1}{n^2} \sum_{j|n} (\varphi(j))^2 \, j^{\frac{n}{j}} \Big( \frac{n}{j} \Big)! \qquad (\text{by} \ \eqref{muu}). \qedhere 
\end{align*}
\end{proof} 

{\tiny
\begin{table} 
\begin{tabular}{r|cl}
\toprule
$n$ & & $T_n$ \\
\midrule 
$2$ & & $1$ \\  
$3$ & & $2$ \\  
$4$ & & $3$ \\  
$5$ & & $8$ \\  
$6$ & & $24$ \\  
$7$ & & $108$ \\  
$8$ & & $640$ \\  
$9$ & & $4492$ \\  
$10$ & & $36336$ \\  
$11$ & & $329900$ \\  
$12$ & & $3326788$ \\  
$13$ & & $36846288$ \\  
$14$ & & $444790512$ \\  
$15$ & & $5811886656$ \\  
\bottomrule 
\end{tabular}
\vspace*{4mm}
\caption{\sl The values of $T_n$ for $2 \leq n \leq 15$.}
\label{teeq}
\end{table}
}

It is evident from Table \ref{teeq} that the sequence $\{ T_n \}$ grows rapidly as $n \to \infty$. More quantitatively, we have the following 

\begin{theorem}\label{asymp}
$T_n \sim (n-2)!$ as $n \to \infty$.
\end{theorem}

\begin{proof}
This is easy to verify. By \eqref{tq},
$$
\frac{n^2 \, T_n}{n!} = 1 + \frac{(\varphi(n))^2}{(n-1)!} + \frac{1}{n!} \sum_j  (\varphi(j))^2 \, j^\frac{n}{j} \Big( \frac{n}{j} \Big)!
$$
where the sum is taken over all divisors $j$ of $n$ with $1<j<n$. We need
only check that the term on the far right tends to $0$ as $n \to \infty$. If $j|n$ and $1<j<n$, then $j \leq \lfloor n/2 \rfloor$ and $n/j \leq \lfloor n/2 \rfloor$. Hence,  
\begin{equation}\label{nina}
(\varphi(j))^2 \, j^{\frac{n}{j}} \Big( \frac{n}{j} \Big)! \leq j^{\frac{n}{j}+2}  \Big( \frac{n}{j} \Big)! \leq \left \lfloor \frac{n}{2} \right \rfloor^{\lfloor n/2 \rfloor +2} \left \lfloor \frac{n}{2} \right \rfloor! 
\end{equation}
The Stirling formula $k! \sim \sqrt{2\pi k} \ k^k \, e^{-k}$ for large $k$ gives the estimate 
$$
\frac{k^k \ k!}{(2k)!} \leq \con  \Big( \frac{e}{4} \Big)^k.
$$
Applying this to \eqref{nina} for $k=\lfloor n/2 \rfloor$, we obtain  
$$
\frac{1}{n!} (\varphi(j))^2 \, j^\frac{n}{j} \Big( \frac{n}{j} \Big)! \leq \con n^2 \Big( \frac{e}{4} \Big)^{\frac{n}{2}}.
$$
This yields the upper bound
$$
\frac{1}{n!} \sum_j (\varphi(j))^2 \, j^\frac{n}{j} \Big( \frac{n}{j} \Big)! \leq \con n^3 \Big( \frac{e}{4} \Big)^{\frac{n}{2}}, 
$$
which tends to $0$ as $n \to \infty$.
\end{proof}

\section{The distribution of degree}\label{sec:des}

We now turn to the problem of enumerating $n$-cycles with a given degree, using the dynamics of a family of linear endomorphisms of the circle. \vs

\noindent
{\it Convention.} We extend the definition of $N_{n,d}$ to all $d \geq 1$ by setting $N_{n,d}=0$ if $d \geq n-1$. 

\subsection{The circle endomorphisms $\mk$}\label{keek}         

For each integer $k \geq 2$, consider the multiplication-by-$k$ map of the circle $\RR/\ZZ$ defined by 
$$
\mk (x):= k x \modd. 
$$ 
Let $\OO = \{ x_1, x_2, \ldots, x_n \}$ be a period $n$ orbit of $\mk$, where the representatives are labeled so that $0 < x_1 < x_2 < \cdots < x_n<1$. We say that $\OO$ {\bit realizes} the cycle $\nu \in \sC_n$ if 
$$
\mk(x_i)=x_{\nu(i)} \qquad \text{for all} \ 1 \leq i \leq n. 
$$  
The orbit $\OO$ is said to realize the combinatorial type $[\nu]$ in $\sC_n$ if it realizes the conjugate cycle $\rho^j \nu \rho^{-j}$ for some $j$. \vs

It follows from the topological interpretation of degree in \S \ref{DNUM} that if an orbit of $\mk$ realizes $\nu \in \sC_{n,d}$, then necessarily $k \geq d$. Conversely, if $\nu \in \sC_{n,d}$ and $k \geq \max\{ d, 2 \}$, there are always periodic orbits of $\mk$ that realize the combinatorial type $[\nu]$. These orbits are essentially determined by how they are deployed on the circle relative to the $k-1$ fixed points $0,1/(k-1),\ldots,(k-2)/(k-1)$ of $\mk$, a useful description that makes it possible to enumerate them effectively. Below is a brief outline of how this is accomplished in \cite{PZ}. \vs

Consider $\{ 1, 2, \ldots, n \}$ with its natural cyclic order $\prec$ on three or more points. We define the {\bit signature} of $\nu \in \sC_{n,d}^s$ as the binary vector $\sig(\nu)=(a_1, \ldots, a_n)$, where  
$$
a_i := \begin{cases} 1 & \quad \text{if} \ \nu(i) \prec i \prec i+1 \prec \nu(i+1)  \\
0 & \quad \text{otherwise.} \end{cases}
$$
The degree $d=\deg(\nu)$ and symmetry order $s=\sym(\nu)$ are both encoded in the signature: There are precisely $d-1$ entries of $1$ in $\sig(\nu)$. Moreover, $\sig(\nu)$ is $r$-periodic, where $r=n/s$: 
\begin{equation}\label{rper}
a_{i+r}=a_i \qquad \text{for all} \ i. 
\end{equation}

To any periodic orbit $\OO=\{ x_1, x_2, \ldots, x_n \}$ of $\mk$ that realizes $\nu$, we assign the cumulative {\bit deployment vector} $\dep(\OO)=(w_1, \ldots, w_{k-1}) \in \ZZ^{k-1}$ whose components are defined by $w_i := \# (\OO \cap (0,i/(k-1)))$. One can check that   
\begin{enumerate}
\item[(i)]
$0 \leq w_1 \leq \cdots \leq w_{k-1}=n$, \vs
\item[(ii)]
$a_i=1$ implies $i \in \{ w_1, \ldots, w_{k-1} \}$. 
\end{enumerate} 
Any integer vector $w=(w_1,\ldots,w_{k-1})$ which satisfies these conditions is called {\bit $\nu$-admissible}.

\begin{theorem}\label{yagh}
Let $\nu \in \sC_{n,d}$ and $k \geq \max\{ d, 2 \}$. For every $\nu$-admissible vector $w \in \ZZ^{k-1}$ there exists a unique periodic orbit $\OO$ of $\mk$ with $\dep(\OO)=w$ that realizes $\nu$.  
\end{theorem} 

See \cite[Theorem 7.9]{PZ}. The idea of the proof is to translate the realization problem to the problem of finding the steady-state of an associated regular Markov chain. The theorem shows that in order to enumerate the orbits of $\mk$ that realize $\nu$ we can simply look at how many $\nu$-admissible vectors in $\ZZ^{k-1}$ there are. A routine count shows that this number is 
$\binom{n+k-d}{n}$ if $a_n=1$ and $\binom{n+k-d-1}{n}$ if $a_n=0$. 

\begin{example}
The $6$-cycle $\nu=(1 \, 3 \, 2 \, 4 \, 6 \, 5)$ has degree $d=3$, symmetry order $s=2$, and signature $\sig(\nu)=(0,0,1,0,0,1)$. For the choice $k=4$, the $\nu$-admissible vectors $(w_1, w_2, w_3) \in \ZZ^3$ are those that satisfy $0 \leq w_1 \leq w_2 \leq w_3=6$ and $\{ 3,6 \} \subset \{ w_1, w_2, w_3 \}$. There are $\binom{6+4-3}{6}=7$ such vectors:
$$
(0,3,6), (1,3,6), (2,3,6), (3,3,6), (3,4,6), (3,5,6), (3,6,6).
$$
Table \ref{tabnew} shows the corresponding $7$ orbits of $\mc$, guaranteed by \thmref{yagh}.
\end{example}  

\begin{table}
\begin{tabular}{c|cl}
\toprule
$\dep(\OO)$ & & $\OO$ \\
\midrule
$(0,3,6)$ & & $\big\{ \frac{183}{455}, \frac{198}{455},  \frac{277}{455}, \frac{337}{455}, \frac{387}{455}, \frac{438}{455} \big\}$ \\[5pt]
$(1,3,6)$ & & $\big\{ \frac{89}{585}, \frac{254}{585}, \frac{356}{585}, \frac{431}{585}, \frac{461}{585}, \frac{554}{585} \big\}$ \\[5pt]
$(2,3,6)$ & & $\big\{ \frac{43}{315}, \frac{58}{315}, \frac{172}{315}, \frac{232}{315}, \frac{247}{315}, \frac{298}{315} \big\}$ \\[5pt]
$(3,3,6)$ & & $\big\{ \frac{101}{1365}, \frac{251}{1365}, \frac{404}{1365}, \frac{1004}{1365}, \frac{1049}{1365}, \frac{1286}{1365} \big\}$ \\[5pt]
$(3,4,6)$ & & $\big\{ \frac{41}{585}, \frac{71}{585}, \frac{164}{585}, \frac{284}{585}, \frac{449}{585}, \frac{551}{585} \big\}$ \\[5pt]
$(3,5,6)$ & & $\big\{ \frac{22}{315}, \frac{37}{315}, \frac{88}{315}, \frac{148}{315}, \frac{163}{315}, \frac{277}{315} \big\}$ \\[5pt]
$(3,6,6)$ & & $\big\{ \frac{94}{1365}, \frac{139}{1365}, \frac{376}{1365}, \frac{556}{1365}, \frac{706}{1365}, \frac{859}{1365} \big\}$ \\[5pt]
\bottomrule
\end{tabular}
\vspace*{2mm}
\caption{\sl Realizations of the $6$-cycle $(1 \, 3 \, 2 \, 4 \, 6 \, 5)$ under the map $\mc$, parametrized by their deployment vectors.}
\label{tabnew}
\end{table}

Now take any $\nu \in \sC_{n,d}^s$ and let $\sig(\nu)=(a_1,\ldots,a_n)$, $r=n/s$. The combinatorial type $[\nu]$ consists of the conjugates $\rho^{-j} \nu \rho^j$ for $1 \leq j \leq r$, and   
$$
\sig(\rho^{-j} \nu \rho^j) = (a_{j+1},\ldots,a_n,a_1,\ldots,a_j).
$$
By the preceding remarks, the number of orbits of $\mk$ that realize $\rho^{-j} \nu \rho^j$ is ${n+k-d \choose n}$ if $a_j=1$ and is ${n+k-d-1 \choose n}$ if $a_j=0$. By $r$-periodicity \eqref{rper}, the list $a_1, \ldots, a_r$ contains $(d-1)/s$ entries of $1$ and $(n-d+1)/s$ entries of $0$. Thus, the number of distinct orbits of $\mk$ that realize $[\nu]$ is
$$
\frac{d-1}{s} {n+k-d \choose n} + \frac{n-d+1}{s} {n+k-d-1 \choose n} = \frac{k-1}{s} {n+k-d-1 \choose n-1}.
$$
This proves the following result (compare \cite[Theorem 7.5]{PZ}):

\begin{theorem}\label{real}
If $\nu \in \sC_{n,d}^s$ and $k \geq \max \{ d,2 \}$, there are precisely 
$$
\frac{k-1}{s} \binom{n+k-d-1}{n-1}  
$$
period $n$ orbits of $\mk$ that realize the combinatorial type $[\nu]$. 
\end{theorem}

The following corollary is immediate: 

\begin{corollary}\label{nib}
For every $k \geq 2$ and $d \geq 1$, the number of period $n$ orbits of $\mk$ that realize some element of $\sC_{n,d}$ is
$$
\frac{k-1}{n} \binom{n+k-d-1}{n-1} N_{n,d}.
$$
\end{corollary}

\begin{proof}
The claim is trivial if $d>k$ since in this case the number of such orbits and the binomial coefficient $\binom{n+k-d-1}{n-1}$ are both $0$. If $2 \leq d \leq k$, then by \thmref{real} for each divisor $s$ of \mbox{$\gcd(n,d-1)$} there are 
$$
\frac{k-1}{s} \binom{n+k-d-1}{n-1} T_{n,d}^s = \frac{k-1}{n} \binom{n+k-d-1}{n-1} N_{n,d}^s
$$
period $n$ orbits of $\mk$ that realize some element of $\sC_{n,d}^s$. The result then follows from \eqref{loo} by summing over all such $s$. Finally, since $\sC_{n,1}^n=\sC_{n,1}$, \thmref{real} shows that there are 
$$
\frac{k-1}{n} \binom{n+k-2}{n-1} T_{n,1} = \frac{k-1}{n} \binom{n+k-2}{n-1} N_{n,1}
$$
period $n$ orbits of $\mk$ that realize some element of $\sC_{n,1}$.  
\end{proof}

\subsection{The numbers $N_{n,d}$}\label{CNQD}

For $k \geq 2$ let $P_n(k)$ denote the number of periodic points of $\mk$ of period $n$. The periodic points of $\mk$ whose period is a divisor of $n$ are precisely the $k^n-1$ solutions of the equation $k^n x = x$ (mod $\ZZ$). Thus,
\begin{equation}\label{masti}
\sum_{r|n} P_r(k) = k^n-1
\end{equation}
and the M\"{o}bius inversion formula gives
\begin{equation}\label{perrr}
P_n(k) = \sum_{r|n} \mu\Big( \frac{n}{r} \Big) (k^r-1).
\end{equation}
Introduce the integer-valued quantity
$$
\Delta_n(k):= \begin{cases} 
\dfrac{P_n(k)}{k-1} & \quad \text{if} \ k \geq 2 \vs \\ 
\varphi(n) & \quad \text{if} \ k=1. \end{cases}
$$
When $k \geq 2$ we can interpret $\Delta_n(k)$ as the number of period $n$ points of $\mk$ up to the rotation of the form $x \mapsto x+j/(k-1) \, (\mo \ZZ)$. This is because $\mk$ and the rotation $x \mapsto x+1/(k-1) \, (\mo \ZZ)$ commute, so $x$ is has period $n$ under $\mk$ if and only if $x+1/(k-1)$ does. \vs  

By \eqref{perrr}, for every $k \geq 2$,
$$
\Delta_n(k) = \sum_{r|n} \mu \Big( \frac{n}{r} \Big) \, \frac{k^r-1}{k-1} = \sum_{r|n} \mu \Big( \frac{n}{r} \Big) \Big(\sum_{j=0}^{r-1} k^j \Big).
$$
If $k=1$, the sum on the far right reduces to $\sum_{r|n} r \mu (n/r)$ which is equal to $\varphi(n)$ by \eqref{muu}. It follows that 
\begin{equation}\label{masq}
\Delta_n(k) = \sum_{r|n} \mu \Big( \frac{n}{r} \Big) \Big(\sum_{j=0}^{r-1} k^j \Big) \qquad \text{for all} \ k \geq 1.
\end{equation}
Since $\mk$ has $P_n(k)/n$ period $n$ {\it orbits} altogether, \corref{nib} shows that for every $k \geq 2$,
$$
\frac{k-1}{n} \sum_{d=1}^{n-2} \binom{n+k-d-1}{n-1} N_{n,d} = \frac{P_n(k)}{n}
$$
or
\begin{equation}\label{mame1}
\sum_{d=1}^{n-2} \binom{n+k-d-1}{n-1} N_{n,d} = \Delta_n(k).
\end{equation}
This is in fact true for every $k \geq 1$ (the case $k=1$ reduces to $N_{n,1}=\Delta_n(1)=\varphi(n)$). 

\begin{remark}\label{ghod}
Since the summand in \eqref{mame1} is zero unless $1 \leq d \leq \min(n-2,k)$, we can replace the upper bound of the sum by $k$. 
\end{remark}

\begin{theorem}\label{jeeg}
For every $d \geq 1$,
\begin{equation}\label{tada}
N_{n,d} = \sum_{i=1}^d (-1)^{d-i} \binom{n}{d-i} \Delta_n(i).
\end{equation}
\end{theorem}

In particular, the theorem claims vanishing of the sum if $d \geq n-1$. Table \ref{tab3} shows the values of $N_{n,d}$ for $2 \leq n \leq 12$. 

{\tiny
\begin{table}
\centering
\begin{tabular}{r|crrrrrrrrrr}
\toprule
\diagbox{$n$}{$d$} & & $1$ & $2$ & $3$ & $4$ & $5$ & $6$ & $7$ & $8$ & $9$ & $10$ \\ 
\midrule
$2$ & & $1$ & & & & & & & & & \\  
$3$ & & $2$ & & & & & & & & & \\  
$4$ & & $2$ & $4$ & & & & & & & & \\  
$5$ & & $4$ & $10$ & $10$ & & & & & & & \\  
$6$ & & $2$ & $42$ & $54$ & $22$ & & & & & &  \\  
$7$ & & $6$ & $84$ & $336$ & $252$ & $42$ & & & & & \\  
$8$ & & $4$ & $208$ & $1432$ & $2336$ & $980$ & $80$ & & & & \\  
$9$ & & $6$ & $450$ & $5508$ & $16548$ & $14238$ & $3402$ & $168$ & & & \\  
$10$ & & $4$ & $950$ & $19680$ & $99250$ & $153860$ & $77466$ & $11320$ & $350$ & & \\ 
$11$ & & $10$ & $1936$ & $66616$ & $534688$ & $1365100$ & $1233760$ & $389224$ & $36784$ & $682$ & \\  
$12$ & & $4$ & $3972$ & $217344$ & $2671560$ & $10568280$ & $15593376$ & $8893248$ & $1851096$ & $116580$ & $1340$ \\ 
\bottomrule  
\end{tabular}
\vspace*{4mm}
\caption{\sl The distributions $N_{n,d}$ for $2 \leq n \leq 12$.}
\label{tab3}
\end{table}
} 

\begin{proof}
This is a form of inversion for binomial coefficients. Use \eqref{mame1} to write 
\begin{align*}
& \hspace{5mm} \sum_{i=1}^d (-1)^{d-i} \binom{n}{d-i} \Delta_n(i) & \\
& = \sum_{i=1}^d \sum_{j=1}^{n-2} (-1)^{d-i} \binom{n}{d-i} \binom{n+i-j-1}{n-1} N_{n,j} & \\
& = \sum_{i=1}^d \sum_{j=1}^i (-1)^{d-i} \binom{n}{d-i} \binom{n+i-j-1}{n-1} N_{n,j} & (\text{by \remref{ghod}}) \\
& = \sum_{j=1}^d \left( \sum_{i=j}^d (-1)^{d-i} \binom{n}{d-i} \binom{n+i-j-1}{n-1} \right) N_{n,j}. & 
\end{align*} 
Thus, \eqref{tada} is proved once we check that 
\begin{equation}\label{saq1}
\sum_{i=j}^d (-1)^{d-i} \binom{n}{d-i} \binom{n+i-j-1}{n-1}  = 0 \qquad \text{for} \ j<d.
\end{equation}  
Introduce the new variables $a:=i-j$ and $b:=d-j>0$ so \eqref{saq1} takes the form   
\begin{equation}\label{saq2}
\sum_{a=0}^b (-1)^a \binom{n}{b-a} \binom{n+a-1}{n-1} = 0.
\end{equation}
The identity 
$$
\binom{n}{b-a} \binom{n+a-1}{n-1} = \frac{n}{b} \ \binom{n+a-1}{b-1} \binom{b}{a}
$$
shows that \eqref{saq2} is in turn equivalent to 
\begin{equation}\label{saq3}
\sum_{a=0}^b (-1)^a  \binom{n+a-1}{b-1} \binom{b}{a} = 0.
\end{equation}
To prove \eqref{saq3}, consider the binomial expansion
$$
P(x):=x^{n-1}(x+1)^b = \sum_{a=0}^b \binom{b}{a} x^{n+a-1}
$$
and differentiate it $b-1$ times with respect to $x$ to get
$$
P^{(b-1)}(x)= (b-1)! \sum_{a=0}^b \binom{n+a-1}{b-1} \binom{b}{a} x^{n+a-b}.
$$ 
Since $P$ has a zero of order $b$ at $x=-1$, we have $P^{(b-1)}(-1)=0$ and \eqref{saq3} follows.   
\end{proof}

As an application of \thmref{jeeg}, we record the following result which will be invoked in \S \ref{sec:stat}:  

\begin{theorem}
The generating function $G_n(x):= \sum_{d=1}^{n-2} N_{n,d} \, x^d$ has the expansion 
\begin{equation}\label{genG}
G_n(x)=(1-x)^n \sum_{i \geq 1} \Delta_n(i) \, x^i.
\end{equation}
\end{theorem} 

This should be viewed as an equality between formal power series. It is a true equality for $|x|<1$ where the series on the right converges absolutely.\footnote{This is because $\Delta_n(i)$ grows like $i^{n-1}$ for fixed $n$ as $i \to \infty$; compare \lemref{boz}.}

\begin{proof}
For each $d \geq 1$ the coefficient of $x^d$ in the product 
$$
(1-x)^n \sum_{i \geq 1} \Delta_n(i) \, x^i = \sum_{j=0}^n (-1)^j \binom{n}{j} x^j \ \cdot \ \sum_{i \geq 1} \Delta_n(i) \, x^i 
$$
is $\sum_{i=1}^d (-1)^{d-i} \binom{n}{d-i} \, \Delta_n(i)$. This is $N_{n,d}$ by \eqref{tada}.
\end{proof}

\subsection{The numbers $T_{n,d}$}\label{ttnn}

Our counts for the numbers $N_n^s$ and $N_{n,d}$ lead to the system of linear equations \eqref{foo} and \eqref{loo} on the $N^s_{n,d}$, but such systems are typically under-determined. Thus, additional information is needed to find the  $N^s_{n,d}$ and therefore $T_{n,d}$. The following example serves to illustrates this point, where we use the dynamics of $\mk$ to obtain this additional information. 

\begin{example}\label{teegg}
For $n=8$ there are nine admissible pairs 
$$
(d,s)= (1,8), (2,1), (3,1), (3,2), (4,1), (5,1), (5,2), (5,4), (6,1)
$$
(see \exref{hasht}). We record the values of $N^s_{8,d}$ on a grid as shown in \figref{t8}. By \eqref{foo} and \eqref{loo}, the values along the $s$-th row add up to $N_8^s$ and those along the $d$-th column add up to $N_{8,d}$, both available from Tables \ref{tab1} and \ref{tab3}. This immediately gives five of the required nine values:
$$
N^8_{8,1}=4, \quad N^1_{8,2}=208, \quad N^1_{8,4}=2336, \quad N^4_{8,5}=4, \quad N^1_{8,6}=80.
$$
\begin{figure}[t]
\centering
\begin{overpic}[width=0.6\textwidth]{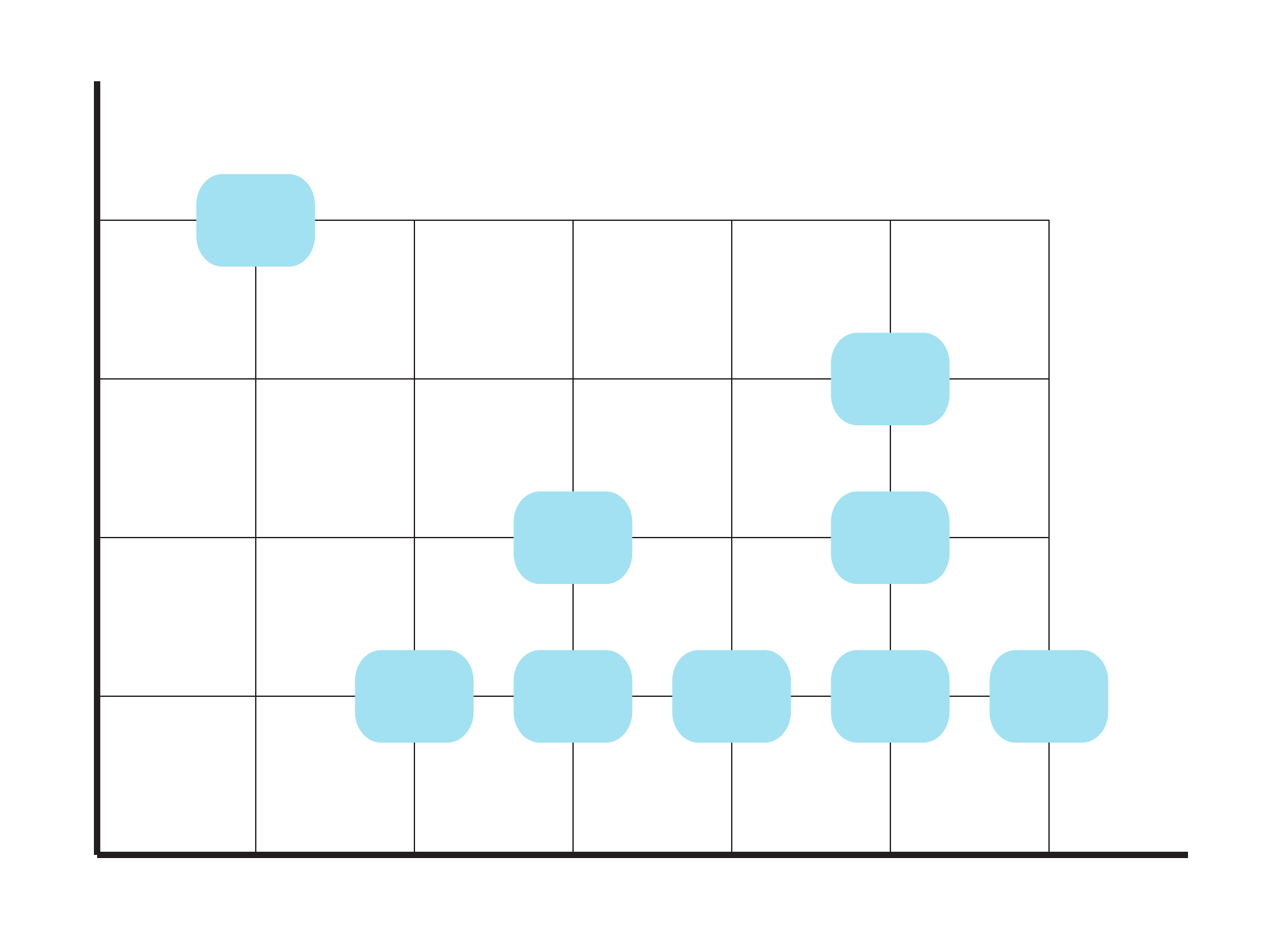}
\put (19,57) {\footnotesize $4$}
\put (29.7,19) {\footnotesize $208$}
\put (41.5,32) {\footnotesize $N^2_{8,3}$}
\put (41.5,19.3) {\footnotesize $N^1_{8,3}$}
\put (53.8,19) {\footnotesize $2336$}
\put (69.4,44.5) {\footnotesize $4$}
\put (66.6,32) {\footnotesize $N^2_{8,5}$}
\put (66.6,19.3) {\footnotesize $N^1_{8,5}$}
\put (81,19) {\footnotesize $80$}
\put (4,19) {\tiny $1$}
\put (4,32) {\tiny $2$}
\put (4,44) {\tiny $4$}
\put (4,57) {\tiny $8$}
\put (19.5,3) {\tiny $1$}
\put (32,3) {\tiny $2$}
\put (44.5,3) {\tiny $3$}
\put (57.2,3) {\tiny $4$}
\put (70,3) {\tiny $5$}
\put (82.5,3) {\tiny $6$}
\put (7,71) {$s$}
\put (96,6.5) {$d$}
\end{overpic}
\caption{\sl Computation of the joint distribution $N^s_{8,d}$ in \exref{teegg}.}  
\label{t8}
\end{figure}
Moreover, it leads to the system of linear equations 
\begin{equation}\label{char}
\begin{cases}
N^1_{8,3}+N^2_{8,3} & \! \! = 1432 \\
N^1_{8,5}+N^2_{8,5} & \! \! = 976 \\
N^1_{8,3}+N^1_{8,5} & \! \! = 2368 \\
N^2_{8,3}+N^2_{8,5} & \! \! = 40
\end{cases}
\end{equation}
on the remaining four unknowns which has rank $3$ and therefore does not determine the solution uniquely. An additional piece of information can be obtained by considering the period $8$ orbits of $\mb$ which realize cycles in $\sC^2_{8,3}$ (see \cite{PZ}, especially Theorem 6.6, for the results supporting the following claims). Every such orbit is {\bit self-antipodal} in the sense that it is invariant under the $180^{\circ}$ rotation $x \mapsto x+1/2$ of the circle $\RR/\ZZ$. It follows that $x$ belongs to such orbit if and only if it satisfies
$$
3^4 x = x+\frac{1}{2} \quad (\operatorname{mod} \ZZ).
$$     
This is equivalent to $x$ being rational of the form  
$$
x=\frac{2j-1}{160} \quad (\operatorname{mod} \ZZ) \quad \text{for some} \ 1 \leq j \leq 80. 
$$
Of the $10$ orbits of $\mb$ thus determined, $4$ realize rotation cycles in $\sC^8_{8,1}$ and the remaining $6$ realize cycles in $\sC^2_{8,3}$. Moreover, by \thmref{real} every combinatorial type in $\sC^2_{8,3}$ is realized by a {\it unique} orbit of $\mb$. It follows that $N^2_{8,3}=4T^2_{8,3}=24$. Now from \eqref{char} we obtain 
$$
N^1_{8,3} = 1408, \quad N^2_{8,3}= 24, \quad 
N^1_{8,5} = 960, \quad N^2_{8,5}= 16 
$$  
and therefore 
$$
T_{8,1}=4, \ \ T_{8,2}=26, \ \ T_{8,3}=182, \ \ T_{8,4}=292, \ \ T_{8,5}=126, \ \ T_{8,6}=10.
$$   
Observe that $T_8=\sum_{d=1}^6 T_{8,d}=640$, in agreement with the value in Table \ref{teeq} coming from formula \eqref{tq}.  
\end{example}

Table \ref{tab4} shows the result of similar but often more complicated dynamical arguments to determine $T_{n,d}$ for $n$ up to $12$. It would be desirable to develop a general method (and perhaps a closed formula) to compute these numbers for arbitrary $n$.  

{\tiny
\begin{table}
\centering
\begin{tabular}{r|crrrrrrrrrr}
\toprule
\diagbox{$n$}{$d$} & & $1$ & $2$ & $3$ & $4$ & $5$ & $6$ & $7$ & $8$ & $9$ & $10$\\ 
\midrule
$2$ & & $1$ & & & & & & & & & \\ 
$3$ & & $2$ & & & & & & & & & \\ 
$4$ & & $2$ & $1$ & & & & & & & & \\  
$5$ & & $4$ & $2$ & $2$ & & & & & & & \\  
$6$ & & $2$ & $7$ & $10$ & $5$ & & & & & & \\  
$7$ & & $6$ & $12$ & $48$ & $36$ & $6$ & & & & & \\  
$8$ & & $4$ & $26$ & $\color{red}182$ & $292$ & $\color{red}126$ & $10$ & & & & \\  
$9$ & & $6$ & $50$ & $612$ & $\color{red}1844$ & $1582$ & $378$ & $\color{red}20$ & & & \\  
$10$ & & $4$ & $95$ & $\color{red}1978$ & $9925$ & $\color{red}15408$ & $7753$ & $\color{red}1138$ & $35$ & & \\
$11$ & & $10$ & $176$ & $6056$ & $48608$ & $124100$ & $112160$ & $35384$ & $3344$ & $62$ & \\ 
$12$ & & $4$ & $331$ & $\color{red}18140$ & $\color{red}222654$ & $\color{red}880848$ & $1299448$ & $\color{red}741260$ & $154258$ & $\color{red}9732$ & $\color{red}113$ \\
\bottomrule 
\end{tabular}
\vspace*{4mm}
\caption{\sl The distributions $T_{n,d}$ for $2 \leq n \leq 12$. The entries in red cannot be obtained from the sole knowledge of the $N_n^s$ and $N_{n,d}$ in Tables \ref{tab1} and \ref{tab3}.}
\label{tab4}
\end{table}
}

\section{A statistical view of the degree}\label{sec:stat}

\subsection{Classical Eulerian numbers}\label{compp}

The numbers $N_{n,d}$ are the analogs of the {\bit Eulerian numbers} $A_{n,d}$ which tally the permutations of descent number $d$ in the full symmetric group $\sS_n$:\footnote{The numbers $A_{n,d}$ are denoted by $\knuth{n}{d}$ in \cite{GKP} and by $A(n,d+1)$ in \cite{St}.} 
$$
A_{n,d} := \# \big\{ \nu \in \sS_n: \des(\nu)=d \big\}. 
$$
For each $n$ the index $d$ now runs from $0$ to $n-1$, with \mbox{$A_{n,0}=A_{n,n-1}=1$}. The Eulerian numbers occur in many contexts, including areas outside of combinatorics, and have been studied extensively (for an excellent account, see \cite{P2}). Here are a few of their basic properties:
\vs
\begin{enumerate}
\item[$\bullet$]
{\it Symmetry}: 
$$
A_{n,d} = A_{n,n-d-1}.
$$ 

\item[$\bullet$]
{\it Linear recurrence relation}: 
$$
A_{n,d} = (d+1) A_{n-1,d} + (n-d) A_{n-1,d-1}.
$$

\item[$\bullet$]
{\it Worpitzky's identity}:
\begin{equation}\label{worp}
\sum_{d=0}^{n-1} \binom{n+k-d-1}{n} A_{n,d} = k^n \qquad \text{for all} \ k \geq 1
\end{equation}

\item[$\bullet$]
{\it Alternating sum formula}:
\begin{equation}\label{asf}
A_{n,d} = \sum_{i=1}^{d+1} (-1)^{d-i+1} \, \binom{n+1}{d-i+1} \, i^n. \vs
\end{equation}

\item[$\bullet$]
{\it Carlitz's identity}: The generating function $A_n(x) := \sum_{d=0}^{n-1} A_{n,d} \, x^d$ (also known as the $n$-th ``Eulerian polynomial'') satisfies
\begin{equation}\label{genA}
A_n(x) = (1-x)^{n+1} \sum_{i \geq 1} i^n \, x^{i-1}. 
\end{equation}  
\end{enumerate}

The last three formulas reveal a remarkable similarity between the sequences $ N_{n,d}$ and $A_{n-1,d-1}$. In fact, \eqref{mame1} is the analog of Worpitzky's identity \eqref{worp} for $A_{n-1,d-1}$ once $\Delta_n(k)$ is replaced with $k^{n-1}$. Similarly, \eqref{tada} is the analog of the alternating sum formula \eqref{asf} for $A_{n-1,d-1}$ when we replace $\Delta_n(i)$ with $i^{n-1}$. Finally, \eqref{genG} is the analog of Carlitz's identity \eqref{genA} for $\sum_{d=1}^{n-1} A_{n-1,d-1} \, x^d = x A_{n-1}(x)$, again replacing $\Delta_n(i)$ with $i^{n-1}$. \vs

There is also an analogy between the $N_{n,d}$ and the restricted Eulerian numbers 
\begin{equation}\label{ren}
B_{n,d} := \# \big\{ \nu \in \sC_n: \des(\nu)=d \big\}.
\end{equation}
In the beautiful paper \cite{DMP} which is motivated by the problem of riffle shuffles of a deck of cards, the authors obtain exact formulas for the distribution of descents in a given conjugacy class of $\sS_n$ (see also \cite{F1} for an alternative approach). As a special case, their formulas show that 
$$
B_{n,d} = \sum_{i=1}^{d+1} (-1)^{d-i+1} \, \binom{n+1}{d-i+1} \, f_n(i),
$$  
where
$$
f_n(i) := \frac{1}{n} \sum_{r|n} \mu \Big( \frac{n}{r} \Big) i^r
$$
is the number of aperiodic circular words of length $n$ from an alphabet of $i$ letters. One cannot help but notice the similarity between the above formula for $B_{n-1,d-1}$ and \eqref{tada}, and between $f_n(i)$ and $\Delta_n(i)$ in \eqref{masq}.  

\subsection{Asymptotic normality}\label{asnorm}

The statistical behavior of classical Eulerian numbers is well understood. For example, it is known that the distribution $\{ A_{n,d} \}_{0 \leq d \leq n-1}$ is unimodal with a peak at $d= \lfloor n/2 \rfloor$. Moreover, the descent number of a randomly chosen permutation in $\sS_n$ (with respect to the uniform measure) has mean and variance 
\begin{align*}
\tilde{\mu}_n & := \frac{1}{n!} \sum_{d=0}^{n-1} d A_{n,d} = \frac{n-1}{2} \\
\tilde{\sigma}^2_n & := \frac{1}{n!} \sum_{d=0}^{n-1} (d-\tilde{\mu}_n)^2 A_{n,d} = \frac{n+1}{12}.
\end{align*}
These computations can be expressed in terms of the generating functions $A_n$ introduced in \S \ref{compp}: 
\begin{align}
\frac{A'_n(1)}{n!} & = \frac{n-1}{2} \label{EN1} \vs \\ 
\frac{A''_n(1)}{n!} + \frac{A'_n(1)}{n!} - \left( \frac{A'_n(1)}{n!} \right)^2 & = \frac{n+1}{12}. \label{VN1}
\end{align}
When normalized by its mean and variance, the distribution $\{ A_{n,d} \}_{0 \leq d \leq n-1}$ converges to the standard normal distribution as $n \to \infty$ (see \cite{B}, \cite{H}, \cite{Pi}). This is the central limit theorem for Eulerian numbers. In fact, we have the error bound
$$
\sup_{x \in \RR} \left| \frac{1}{n!} \sum_{d \leq \tilde{\sigma}_n x + \tilde{\mu}_n} A_{n,d} - \frac{1}{\sqrt{2\pi}} \int_{-\infty}^x e^{-t^2/2} \, dt \right| = O(n^{-1/2}).
$$ 

Similar results hold for the restricted Eulerian numbers $B_{n,d}$ defined in \eqref{ren}. In \cite{F1}, Fulman shows that the mean and variance of $\des(\nu)$ for a randomly chosen $\nu \in \sC_n$ are also $(n-1)/2$ and $(n+1)/12$ provided that $n \geq 3$ and $n \geq 4$ respectively. More generally, he shows that the $k$-th moment of $\des(\nu)$ for $\nu \in \sC_n$ is equal to the $k$-th moment of $\des(\nu)$ for $\nu \in \sS_n$ provided that $n \geq 2k$. From this result one can immediately conclude that the normalized distribution $B_{n,d}$ is also asymptotically normal as $n \to \infty$. \vs  

Below we will prove corresponding results for the distribution of degree for randomly chosen $n$-cycles.  

\begin{theorem}\label{exp}
The mean and variance of $\deg(\nu)$ for a randomly chosen $\nu \in \sC_n$ (with respect to the uniform measure) are 
\begin{align*} 
\mu_n & := \frac{1}{(n-1)!} \sum_{d=1}^{n-2} d N_{n,d} = \frac{n}{2}-\frac{1}{n-1} & (n \geq 3), \\
\sigma^2_n & := \frac{1}{(n-1)!} \sum_{d=1}^{n-2} (d-\mu_n)^2 N_{n,d} = \frac{n}{12}+\frac{n}{(n-1)^2(n-2)} & (n \geq 5).
\end{align*}
\end{theorem}

\begin{proof}
The argument is inspired by the method of \cite[Theorem 2]{F1}. We begin by using  the formula \eqref{masq} for $\Delta_n(i)$ in the equation \eqref{genG} to express the generating function $G_n$ in terms of the Eulerian polynomials $A_j$ in \eqref{genA}: 
\begin{align}\label{divo}
G_n(x) & = (1-x)^n \sum_{i \geq 1} \sum_{r|n} \sum_{j=0}^{r-1} \mu \Big( \frac{n}{r} \Big) \, i^j x^i \notag \\
& = (1-x)^n \sum_{i \geq 1} \sum_{j=0}^{n-1} i^j x^i + (1-x)^n \sum_{i \geq 1} \sum_{\substack{r|n \\ r<n}} \sum_{j=0}^{r-1} \mu \Big( \frac{n}{r} \Big) \, i^j x^i \notag \\
& = \sum_{j=0}^{n-1} x(1-x)^{n-j-1} A_j(x) + \sum_{\substack{r|n \\ r<n}} \sum_{j=0}^{r-1} \mu \Big( \frac{n}{r} \Big) \, x(1-x)^{n-j-1} A_j(x). 
\end{align}
If $n \geq 3$, every index $j$ in the double sum in \eqref{divo} is $\leq n-3$, so the polynomial in $x$ defined by this double sum has $(1-x)^2$ as a factor. It follows that for $n \geq 3$,
$$
G_n(x)=x A_{n-1}(x)+ x(1-x) A_{n-2}(x) + O((1-x)^2)
$$ 
as $x \to 1$. This gives
$$
G'_n(1) = A'_{n-1}(1)+A_{n-1}(1)-A_{n-2}(1), 
$$
so by \eqref{EN1}
$$
\mu_n = \frac{G'_n(1)}{(n-1)!} = \frac{n-2}{2}+1-\frac{1}{n-1}= \frac{n}{2}-\frac{1}{n-1}.
$$
Similarly, if $n \geq 5$, every index $j$ in the double sum in \eqref{divo} is $\leq n-4$, so the polynomial defined by this double sum has $(1-x)^3$ as a factor. It follows that for $n \geq 5$,
$$
G_n(x)=x A_{n-1}(x)+ x(1-x) A_{n-2}(x) + x(1-x)^2 A_{n-3}(x) + O((1-x)^3) 
$$ 
as $x \to 1$. This gives
$$
G''_n(1) = A''_{n-1}(1)+2A'_{n-1}(1)-2A'_{n-2}(1)-2A_{n-2}(1)+2A_{n-3}(1). 
$$
A straightforward computation using \eqref{EN1} and \eqref{VN1} then shows that 
\[
\sigma^2_n = \frac{G''_n(1)}{(n-1)!} + \frac{G'_n(1)}{(n-1)!} - \left( \frac{G'_n(1)}{(n-1)!} \right)^2 = \frac{n}{12} + \frac{n}{(n-1)^2(n-2)}. \qedhere 
\] 
\end{proof}

\begin{remark}
More generally, the expression \eqref{divo} shows that for fixed $k$ and large enough $n$,
$$
G_n(x)=\sum_{j=0}^k x(1-x)^j A_{n-j-1}(x) + O((1-x)^{k+1})
$$
as $x \to 1$. Differentiating this $k$ times and evaluating at $x=1$, we obtain the relation
$$
G_n^{(k)}(1)= \sum_{j=0}^k (-1)^j \left( \binom{k}{j} j! \ A^{(k-j)}_{n-j-1}(1) + \binom{k}{j+1} (j+1)! \ A^{(k-j-1)}_{n-j-1}(1) \right)
$$
which in theory links the moments of $\deg(\nu)$ for $\nu \in \sC_n$ to the moments of $\des(\nu)$ for $\nu \in \sS_j$ for $n-k \leq j \leq n-1$. 
\end{remark}

Numerical evidence suggest that the distribution $\{ N_{n,d} \}_{1 \leq d \leq n-2}$ is also unimodal and reaches a peak at $d= \lfloor n/2 \rfloor$. \thmref{asnor} below asserts that when normalized by its mean and variance, the distribution $\{ N_{n,d} \}_{1 \leq d \leq n-2}$ converges to normal as $n \to \infty$. In particular, the asymmetry of the numbers $N_{n,d}$ relative to $d$ will asymptotically disappear. These facts are illustrated in \figref{normal}. \vs

Consider the sequence of normalized random variables 
$$
X_n := \frac{1}{\sigma_n} \big( \deg|_{\sC_n} - \mu_n \big).
$$
Let ${\mathcal N}(0,1)$ denote the normally distributed random variable having  mean $0$ and variance $1$.

\begin{theorem}\label{asnor}
$X_n \to {\mathcal N}(0,1)$ in distribution as $n \to \infty$.   
\end{theorem}

The proof follows the strategy of \cite{FKL} and makes use of the following variant of a classical theorem of Curtiss. Recall that the {\bit moment generating function} $M_X$ of a random variable $X$ is the expected value of $e^{sX}$:
$$
M_X(s):={\mathbb E}(e^{sX}) \qquad \qquad (s \in \RR).
$$

\begin{lemma}[\cite{KL2}]\label{kl}
Let $\{ X_n \}_{n \geq 1}$ and $Y$ be random variables and assume that $\lim_{n \to \infty} M_{X_n}(s)=M_{Y}(s)$ for all $s$ in some non-empty open interval in $\RR$. Then $X_n \to Y$ in distribution as $n \to \infty$.  
\end{lemma}

The proof of \thmref{asnor} via \lemref{kl} will depend on two preliminary estimates.

\begin{lemma}\label{boz}
For every $\ve>0$ there are constants $n(\ve), i(\ve)>0$ such that 
$$
\Delta_n(i) \begin{cases} \ \leq (1+\ve) \, i^{n-1} & \text{if} \ n \geq 2 \ \text{and} \ i \geq i(\ve) \vs \\
\ \geq (1-\ve) \, i^{n-1} & \text{if} \ n \geq n(\ve) \ \text{and} \ i \geq 2.
\end{cases}
$$
\end{lemma}

\begin{proof}
By \eqref{masti},
$$
\Delta_n(i) \leq \sum_{r|n} \Delta_r(i) = \frac{i^n-1}{i-1}.
$$
The upper bound follows since $(i^n-1)/(i-1)<(1+\ve) i^{n-1}$ for all $n$ if $i$ is large enough depending on $\ve$. \vs
 
For the lower bound, first note that the inequality $(i^r-1)/(i-1) \leq 2 i^{r-1}$ holds for all $r \geq 1$ and all $i \geq 2$. Thus, by \eqref{masq}, we can estimate
\begin{align*}
\Delta_n(i) & \geq \frac{i^n-1}{i-1} - \sum_{\substack{r|n \\ r<n}}  \frac{i^r-1}{i-1} \geq i^{n-1} - \sum_{\substack{r|n \\ r<n}} 2i^{r-1} \\
& \geq i^{n-1} - 2 \sum_{r=1}^{\lfloor n/2 \rfloor} i^{r-1} \geq i^{n-1} - 2 \ \frac{i^{n/2}-1}{i-1} \\
& \geq i^{n-1} - 4 i^{n/2-1}.
\end{align*}
The last term is bounded below by $(1-\ve) i^{n-1}$ for all $i$ if $n$ is large enough depending on $\ve$.   
\end{proof}

\begin{lemma}[\cite{FKL}]\label{jag}
For every $0<x<1$ and $n \geq 1$,
$$
\frac{(n-1)! \, x}{(\log(1/x))^n} \leq \sum_{i \geq 2} i^{n-1} x^i \leq \frac{(n-1)!}{x (\log(1/x))^n}. 
$$
\end{lemma}

\begin{proof}
By elementary calculus, 
$$
\sum_{i \geq 2} i^{n-1} x^i \leq \int_0^{\infty} u^{n-1} x^{u-1} \, du = \frac{(n-1)!}{x (\log(1/x))^n} 
$$
and 
\[
\sum_{i \geq 2} i^{n-1} x^i \geq \int_0^{\infty} u^{n-1} x^{u+1} \, du = \frac{(n-1)! \, x}{(\log(1/x))^n}. \qedhere 
\]
\end{proof}

\begin{proof}[Proof of \thmref{asnor}]
By \lemref{kl} it suffices to show that $\lim_{n \to \infty} M_{X_n}(s) = M_{{\mathcal N}(0,1)}(s)=e^{s^2/2}$ for all negative values of $s$. Fix an $s<0$ and set $0<x := e^{s/\sigma_n}<1$ (we will think of $x$ as a function of $n$, with $\lim_{n \to \infty} x=1$). Notice that by \thmref{exp}
\begin{equation}\label{khak}
\mu_n = \frac{n}{2} + O(n^{-1}) \quad \text{and} \quad \sigma^2_n = \frac{n}{12} + O(n^{-2}) \quad \text{as} \ n \to \infty.
\end{equation}
Using \eqref{genG}, we can write 
\begin{align*}
M_{X_n}(s) & = {\mathbb E}(e^{sX_n}) = \frac{e^{-s \mu_n/\sigma_n}}{(n-1)!} \, G_n(e^{s/\sigma_n}) = \frac{x^{-\mu_n}}{(n-1)!} \,  G_n(x) \\
& = \frac{x^{1-\mu_n}(1-x)^n  \varphi(n)}{(n-1)!} +  \frac{x^{-\mu_n}(1-x)^n}{(n-1)!} \sum_{i \geq 2} \Delta_n(i) x^i.
\end{align*}
As the first term is easily seen to tend to zero, it suffices to show that 
\begin{equation}\label{laboo}
H_n := \frac{x^{-\mu_n}(1-x)^n}{(n-1)!} \sum_{i \geq 2} \Delta_n(i) x^i \xrightarrow{n \to \infty} e^{s^2/2}.
\end{equation}
By \eqref{khak} we have the estimate  
$$
1-x = -\frac{s}{\sigma_n} - \frac{s^2}{2 \sigma^2_n} + O(n^{-3/2}). 
$$
This, combined with the basic expansion
$$
\log \left(\frac{1-x}{\log(1/x)} \right) = -\frac{1}{2} (1-x) - \frac{5}{24} (1-x)^2 + O((1-x)^3),
$$
shows that 
\begin{equation}\label{hulu}
\left(\frac{1-x}{\log(1/x)} \right)^n = \exp \left( \frac{ns}{2\sigma_n} + \frac{n s^2}{24 \sigma^2_n} + O(n^{-1/2}) \right). 
\end{equation}
Take any $\ve>0$ and find $n(\ve)$ from \lemref{boz}. Then, if $n \geq n(\ve)$, 
\begin{align*}
H_n & \geq \frac{x^{-\mu_n}(1-x)^n}{(n-1)!} \, (1-\ve) \sum_{i \geq 2} i^{n-1} x^i & \\
& \geq (1-\ve) \, x^{1-\mu_n} \left(\frac{1-x}{\log(1/x)} \right)^n & (\text{by \lemref{jag}}) \\
& = (1-\ve) \exp \left( \frac{s(1-\mu_n)}{\sigma_n} +\frac{ns}{2\sigma_n} + \frac{n s^2}{24 \sigma^2_n} + O(n^{-1/2}) \right) & (\text{by \eqref{hulu}}) \\
& = (1-\ve) \exp \left( \frac{s(1+O(n^{-1}))}{\sigma_n} + \frac{s^2}{2+O(n^{-3})} + O(n^{-1/2}) \right) & (\text{by \eqref{khak}}). 
\end{align*}
Taking the $\liminf$ as $n \to \infty$ and then letting $\ve \to 0$, we obtain
$$
\liminf_{n \to \infty} H_n \geq e^{s^2/2}.
$$
Similarly, take any $\ve>0$, find $i(\ve)$ from \lemref{boz} and use the basic inequality $\Delta_n(i) \leq (i^n-1)/(i-1) \leq 2i^{n-1}$ for all $n,i \geq 2$ to estimate 
\begin{align*}
H_n & = \frac{x^{-\mu_n}(1-x)^n}{(n-1)!} \left( \sum_{2 \leq i <i(\ve)} + \sum_{i \geq i(\ve)} \right) \Delta_n(i) x^i \\
& \leq \frac{2x^{-\mu_n}(1-x)^n}{(n-1)!} \sum_{2 \leq i <i(\ve)} i^{n-1} x^i +  \frac{(1+\ve) x^{-\mu_n}(1-x)^n}{(n-1)!} \sum_{i \geq i(\ve)} i^{n-1} x^i. 
\end{align*}
The first term is a polynomial in $x$ and is easily seen to tend to zero as $n \to \infty$. The second term is bounded above by  
\begin{align*}
& (1+\ve) \, x^{-1-\mu_n} \left(\frac{1-x}{\log(1/x)} \right)^n & (\text{by \lemref{jag}}) \\
= & (1+\ve) \exp \left( \frac{s(-1-\mu_n)}{\sigma_n} + \frac{ns}{2\sigma_n} + \frac{n s^2}{24 \sigma^2_n} + O(n^{-1/2}) \right) & (\text{by \eqref{hulu}}) \\  
= & (1+\ve) \exp \left( \frac{s(-1+O(n^{-1}))}{\sigma_n} + \frac{s^2}{2+O(n^{-3})} + O(n^{-1/2}) \right) & (\text{by \eqref{khak}})
\end{align*}
Taking the $\limsup$ as $n \to \infty$ and then letting $\ve \to 0$, we obtain
$$
\limsup_{n \to \infty} H_n \leq e^{s^2/2}.
$$
This verifies \eqref{laboo} and completes the proof. 
\end{proof}

\end{document}